\documentclass{amsart}
\usepackage[T1]{fontenc}
\usepackage[utf8]{inputenc}
\usepackage[pagebackref]{hyperref}
\hypersetup{pdfpagemode=FullScreen,colorlinks=false}
\usepackage{microtype,amsfonts,amssymb,mathrsfs}
\usepackage[alphabetic]{amsrefs}
\usepackage[scr=rsfs]{mathalfa}
\usepackage{tikz-cd}
\usepackage{stmaryrd}

\theoremstyle{plain}
\newtheorem{theorem}[equation]{Theorem}
\newtheorem{lemma}[equation]{Lemma}

\newtheorem*{theorem*}{Theorem}
\newtheorem{corollary}[equation]{Corollary}
\theoremstyle{definition}
\newtheorem{definition}[equation]{Definition}
\newtheorem{instance}[equation]{Example}
\newtheorem{situation}[equation]{}
\newtheorem*{situation*}{}
\newtheorem{remark}[equation]{Remark}
\usepackage{enumitem}
\setlist[itemize]{leftmargin=2em}
\numberwithin{equation}{section}


\title[Relative cohomology]{\(A\)-hypergeometric systems and relative cohomology}

\author{Tsung-Ju Lee}
\address{T.-J.~Lee: Center of Mathematical Sciences and Applications, 20 Garden St., Cambridge, MA 02138.}
\email{tjlee@cmsa.fas.harvard.edu}

\author{Dingxin Zhang}
\address{D.~Zhang: Yau Mathematical Sciences Center, Tsinghua University, Beijing 100084, China.}
\email{dingxin@tsinghua.edu.cn}

\begin{document}
\begin{abstract}
We investigate the solution space to certain \(A\)-hypergeometric
\(\mathscr{D}\)-modules, which were defined and studied by Gelfand,
Kapranov, and Zelevinsky. We show that the solution space can be
identified with certain relative cohomology group of the toric
variety determined by \(A\), which generalizes the results of
Huang, Lian, Yau, and Zhu. As a corollary, we also prove the
existence of rank one points for complete intersections
in toric varieties.
\end{abstract}

\maketitle

\section{Introduction}
An \(A\)-hypergeometric system, or a GKZ-system, is a certain holonomic system
of linear partial differential equations determined by some combinatoric data. This
notion was introduced by
Gelfand--Kapranov--Zelevinsky~\cite{1989-Gelfand-Kapranov-Zelebinski-hypergeometric-functions-and-toral-manifolds}.
It has a variety of applications in algebra, algebraic geometry, and mirror
symmetry.

The purpose of this note is to give a cohomological interpretation of the space
of local solutions to certain \(A\)-hypergeometric systems related to ``period integrals''. Our theorem
generalizes a result of
Huang--Lian--Yau--Zhu~\cite{2016-Huang-Lian-Yau-Zhu-chain-integral-solutions-to-tautological-systems}.
The theorem of Huang et.~al.~states that, for the special \(A\)-hypergeometric
system associated with period integrals of \emph{Calabi--Yau} hypersurfaces
(see Remark~\ref{remark:toric-cy-complete-intersections}) in a \emph{smooth}
projective toric variety, its solution space near a point \(x\) can be
identified with a relative homology group
\(\mathrm{H}_n(X\setminus{Y}_{x},(X\setminus Y_{x}) \cap D)\), where \(D\) is
the union of all torus-invariant divisors of \(X\), and \(Y_{x}\) is a (possibly
singular) Calabi--Yau hypersurface associated with the point \(x\).

In this note, we relax both the Calabi--Yau condition and the
smoothness assumption. We also give a uniform treatment to
\(A\)-hypergeometric systems associated with complete intersections in a
projective toric variety. In terms of the geometry of toric varieties, a special
case of our result can be stated as follows.

\begin{theorem*}
Let \((X,L)\) be a polarized toric variety of dimension \(n\) defined by a
convex polytope \(\Delta \subset \mathbb{Z}^{n}\) containing \(\mathbf{0}\) as its interior point. Let
\(L=L_1\otimes\cdots\otimes L_{r}\) be a base-point-free decomposition of \(L\). We consider the
\(A\)-hypergeometric system \(M_{A,\beta_{0}}\) on
\(V=\prod_{m=1}^{r}\mathrm{H}^{0}(X,L_{m})\) determined by ``period integrals''
\begin{equation*}
\int (f_{1}\cdots f_{r})^{-1} \frac{\mathrm{d}t_1}{t_1} \wedge \cdots \wedge \frac{\mathrm{d}t_n}{t_n}
\end{equation*}
where \(f_i\) is the function on \((\mathbb{C}^{\ast})^{n}\) coming from a section of \(L_i\).
Then for any \(x \in V\), we have
\begin{equation*}
  \mathrm{Sol}^{0}(M_{A,\beta_0}) = \mathrm{H}_{n}(X \setminus Y_{x}, (X\setminus Y_{x}) \cap D).
\end{equation*}
Here, \(D\) is the union of all the torus-invariant divisors; for \(x=(f_1,\ldots,f_r)\), \(Y_{x} = \cup_{m=1}^{r} \overline{(f_m=0)}\);
\(\mathrm{Sol}^{0}\) is the classical solution functor.
\end{theorem*}

As an application, we show the existence of a ``rank 1 point'' (which was called
a \emph{maximal degeneracy point}
in~\cite{1997-Hosono-Lian-Yau-maximal-degeneracy-points-of-gkz-systems}) for any
family of Calabi--Yau complete intersections in any Gorenstein Fano toric
variety, generalizing the theorem of
Hosono--Lian--Yau~\cite{1997-Hosono-Lian-Yau-maximal-degeneracy-points-of-gkz-systems}.
See Corollary~\ref{situation:corollaries}.

The assumption that \(\mathbf{0}\) is an interior point of \(\Delta\) can not be
removed from the hypotheses. This condition is also crucial for the existence of
rank one point, Corollary~\ref{situation:corollaries}.

Our main theorem,~Theorem~\ref{eq:main}, can be applied to a more general
situation than the one mentioned above.
See Example~\ref{example:simplified-moduli} for such an
application, where we deal with the \(A\)-hypergeometric system on the
``simplified moduli space''.
See Situation~\ref{situation:toric-notation} and
Hypothesis~\ref{situation:hypothesis} for the
conditions on the \(A\)-hypergeometric system in order to ensure the validity of
Theorem~\ref{eq:main}.

Even if one is only interested in Calabi--Yau spaces, our theorem is desirable,
because most toric varieties one meets in applications and computations are
singular. For example, the hypotheses of the theorem of Huang~et.~al.~are not
satisfied by the the projective toric variety defined by the convex hull of
\begin{equation*}
\{(\delta_{1,i},\ldots,\delta_{n,i}):i=1,\ldots,n\}\cup\{(-1,\ldots,-1)\} \subset \mathbb{Z}^{n}
\end{equation*}
(sometimes referred to as the ``mirror \(\mathbb{P}^{n}\)'').
Thus, one already needs the full strength of our theorem \eqref{eq:main} to
relate the solutions to the \(A\)-hypergeometric system and the cohomological
objects attached to the so-called ``mirror Calabi--Yau spaces'', the simplest
type of Calabi--Yau spaces.

The note is organized as follows. Section~\ref{sec:a-hyp} contains a more
in-depth introduction to our theorem and the proof of the application on the
existence of rank 1 point in some cases. Section~\ref{sec:semi-nonresonant}
proves our situation fits into a picture studied by
Reichelt~\cite{2014-Reichelt-laurent-polynomials-GKZ-hypergeometric-systems-and-mixed-hodge-modules},
so that the \(A\)-hypergeometric system we are interested in is a Fourier
transform of an extraordinary Gauss--Manin system. This is the key input for the
rest of the paper. After reviewing some basics
about \(\mathscr{D}\)-modules in Section~\ref{sec:d-mod}, we show how to use exponential
twists to compute cohomology of various spaces in Section~\ref{sec:cohomology}. Section~\ref{sec:transition}
proves a lemma relating Fourier transform of an extraordinary
Gauss--Manin system and an exponentially twisted \(\mathscr{D}\)-module. The
last section finishes the proof of the main theorem.

\subsection*{Acknowledgement}
We thank An~Huang, Bong~Lian, Shing-Tung~Yau and
Chenglong~Yu for introducing this question to us, and for the conversations they
had with us. We are grateful to Jie~Zhou, who explained to us his related work
on the Hesse pencil of elliptic curves and the historic background of chain integral
formulae.

T.-J.~Lee is grateful to Professors Lian and Yau for the invitation to Harvard
CMSA in April 2018, where this work was partially done.

D.~Zhang is grateful to Baohua Fu and Xuanyu Pan for their hospitality during
his visit to Academy of Mathematics and System Science in July 2018.

\section{\texorpdfstring{\(A\)}{A}-hypergeometric systems}
\label{sec:a-hyp}

We first introduce, in Situation~\ref{situation:toric-notation}, the necessary
notation for the definition of the \(A\)-hypergeometric system we are interested
in. Remark~\ref{remark:toric-cy-complete-intersections} will explain how to
deduce a collection of data satisfying Situation~\ref{situation:toric-notation}
from toric geometry. After explaining a technical hypothesis which has a natural
geometric meaning in~\ref{situation:hypothesis}, we state the main theorem of
this note.

\begin{situation}\label{situation:toric-notation}
  We recall the definition of \(A\)-hypergeometric systems for complete
  intersections. Fix an integer \(r>0\).
  \begin{enumerate}
	\item Let \(V_i = \mathbb{C}^{N_i}\) be complex vector spaces of dimension
    \(N_i\), \(i=1,\ldots,r\). Set \(N=N_1+\cdots+N_r\) and
    \(V=V_1\times\cdots\times V_r\).
	\item Let \(x_{i,1},\cdots, x_{i,N_i}\) be a fixed coordinate system on the
    \emph{dual} vector space \({V_i}^{\vee}\). Set
    \(\partial_{i,j}=\partial/\partial x_{i,j}\),~\(1\le j\le N_i\).
	\item Let
    \(\{(s_1,\ldots,s_r,t_1,\ldots,t_n): s_j,~t_i \in \mathbb{C}^{\ast}\}\) be an algebraic
    torus of dimension \(r+n\). We will write this torus as a product
    \((\mathbb{C}^{\ast})^{r}\times T\) with \(\dim T = n\)
	\item Let \(A\) be an \((r+n) \times N\) matrix with integral entries. We shall
    write \(A=[A_1,\ldots,A_r]\), where \(A_i\) is an \((r+n)\times N_i\) matrix. We
    write \(A_i\) into columns
    \begin{equation*}
      A_i=\begin{bmatrix}
        a_{i,1} & \cdots & a_{i,N_i}\\
        w_{i,1} & \cdots & w_{i,N_i}
      \end{bmatrix}
      =\begin{bmatrix}
        \mu_{i,1} & \cdots & \mu_{i,N_i}
      \end{bmatrix}
    \end{equation*}
    where \(a_{i,j}\in\mathbb{Z}^r\) and \(\mu_{i,j}\in\mathbb{Z}^{r+n}\).
    We also assume that \(\mathrm{rank}(A)=r+n\).
	\item Let \(\tau_i: (\mathbb{C}^{\ast})^r \times T \to V_i\) be the map defined by the submatrix \(A_i\)
    \begin{equation*}
      (s,t) \mapsto (s^{a_{i,1}}t^{w_{i,1}}, \ldots, s^{a_{i,N_i}}t^{w_{i,N_i}})
    \end{equation*}
    and \(\bar{\tau_i}\) be the composition \((\mathbb{C}^{\ast})^{r} \times T \to V_{i}\setminus \{0\}\to \mathbb{P}V_i\).
    Let \(\tau=(\tau_1,\ldots,\tau_r)\) and \(\bar{\tau}=(\bar{\tau}_1,\ldots,\bar{\tau}_r)\).
    We assume that \(\tau\) is \emph{injective}. (This is
    not a serious restriction, but is required by Reichelt's
    theorem~\ref{eq:reichelt-theorem}. We show how to get rid of it in
    Lemma~\ref{lemma:injective-gkz}).
	\item We assume \(a_{i,j}=(\delta_{i,1},\ldots,\delta_{i,r})^t\) for \(j=1,\ldots,N_i\).
    Here \(\delta_{i,k}\) is the Kronecker delta.
	\item Let \(X'\) be the closure of the image of \(\bar{\tau}\).
    \(X'\) is a toric variety with maximal torus \(T'=\bar{\tau}(T)\).
    Let \(\mathcal{L}_{i}^{\vee}\) be the pullback of
    the very ample line bundle on \(\mathbb{P}V_{i}\) to \(X'\).
    Then for each \(x=(x_1,\ldots,x_r)\in V^{\vee}=\prod_{i=1}^{r} V_{i}^{\vee}\),
  	we denoted by \(Y'_{x_i}\) the subvariety in \(X'\) defined by \(x_{i}=0\)
  	(cf. Remark~\ref{remark:toric-cy-complete-intersections}).
  	We also set \(Y'_x:=\cup_{i=1}^{r} Y'_{x_i}\) and \(U'_{x}:=X'-Y'_{x}\).
	\item Let \(X\to X'\) be a toric resolution of singularities.
    Denote by \(U_{x}\) and \(Y_{x}\) the preimage of \(U_{x}'\) and \(Y_{x}'\) inside \(X\) respectively.
  \end{enumerate}
\end{situation}

  Given \(\beta \in \mathbb{C}^{r+n}\), the \(A\)-hypergeometric ideal \(I_{A,\beta}\)
  is the left ideal of the Weyl algebra \(\mathscr{D}=\mathbb{C}[x,\partial]\) on
  the \emph{dual} vector space \(V^{\vee}\) generated by the following two types of
  operators

  \begin{itemize}
    \itemsep=3pt
  \item The ``box operators'': \(\partial^{\nu_+} - \partial^{\nu_{-}}\),
    where \(\nu_{\pm}\in \mathbb{Z}_{\geq 0}^{N}\) satisfy \(A\nu_+=A\nu_{-}\).
    Here for \(m \in \mathbb{Z}_{\geq 0}^{N}\) we write
    \(\partial^m = \partial_{1,1}^{m_{1,1}} \cdots \partial_{r,N_r}^{m_{r,N_r}}\).
  \item The ``Euler operators'': \(E_{l} - \beta_l\), where
    \(E_l=\sum_{i,j}\langle \mu_{i,j},\mathrm{e}_l\rangle x_{i,j}\partial_{i,j}\).
    Here \(\mathrm{e}_l=(\delta_{l,1},\ldots,\delta_{l,n+r})\in\mathbb{Z}^{r+n}\).
  \end{itemize}
  We shall only be interested in a special \(\beta\), namely
  \begin{equation*}
    \beta_0 = (-1,\ldots,-1,0,\ldots,0) \in \mathbb{C}^{r} \times \mathbb{C}^{n}.
  \end{equation*}
  The \(A\)-hypergeometric system \(M_{A,\beta}\) is the cyclic \(\mathscr{D}\)-module
  \(\mathscr{D}/I_{A,\beta}\). As shown by
  Gelfand~et.~al.~\cite{1989-Gelfand-Kapranov-Zelebinski-hypergeometric-functions-and-toral-manifolds},
  (under our hypothesis on \(a_{i,j}\)) and
  Adolphson~\cite{1994-Adolphson-hypergemetric-functions-and-rings-generated-by-monomials}
  (in general), \(M_{A,\beta}\) is a holonomic \(\mathscr{D}\)-module.

\begin{remark}
  \label{remark:toric-cy-complete-intersections}
  The \(\mathscr{D}\)-module \(M_{A,\beta_0}\) naturally arises from the study of
  period integrals on toric varieties. Let \((X,L)\)
  be an \(n\)-dimensional polarized toric variety defined by a full-dimensional polytope
  \(\Delta\subset \mathbb{Z}^{n}\). We assume that \(\mathbf{0}\) is an interior
  point of \(\Delta\). In particular, the section of \(L\) corresponding to
  \(\mathbf{0}\) is supported on the union of all the torus-invariant divisors.

  Suppose further that there exists a collection of base-point-free invertible
  sheaves \(L_i\) on \(X\) such that \(L=L_1 \otimes \cdots \otimes L_{r}\).
  Then the section polytope \(\Delta_i\) of \(L_i\) satisfies the condition that
  \(\mathbf{0}\in \Delta_i\) (since \(L_i\) contains some torus-invariant
  divisor as its section), and \(\Delta=\Delta_1+\cdots+\Delta_r\).

  By assumption, each \(L_i\) determines a morphism
  \(X\to\mathbb{P}(V_i)\) with \(V_i=\mathrm{H}^0(X,L_i)^{\vee}\).
  Put \(V=V_1\times\cdots\times V_r\).
  A generic element \(x=(x_1,\ldots,x_r)\in V^{\vee}\) determines a smooth
  complete intersection \(\cap_{i=1}^{r} Y_{x_i}\)
  in \(X\). We can write
  \begin{equation*}
   \sigma_i := x_{i}|_{(\mathbb{C}^{\ast})^{n}}=\sum_{w_{i,j}\in\Delta_i\cap\mathbb{Z}^n} x_{i,j} t^{w_{i,j}}
  \end{equation*}
  (recall that \((t_1,\ldots,t_n)\) is the coordinate of the embedded torus).
  These data, together with \(a_{i,j}=(\delta_{i,1},\ldots,\delta_{i,r})^t\),
  \(1\le i\le r\) and \(1\le j\le N_i:=\dim V_i\), gives a matrix \(A\) which
  meets the requirements in
  Situation~\ref{situation:toric-notation}.

  The \(A\)-hypergeometric system \(M_{A,\beta_0}\) associated with the matrix
  \(A\) just constructed then contains ``period integrals'' as its solutions.
  On the open embedded torus \((\mathbb{C}^{\ast})^{n}\), we can consider
  \begin{equation}\label{equation:complete:intersection:toric:cy2}
    \int_{\tilde{C}} \frac{1}{\sigma_1\cdots\sigma_r}
    \frac{dt_1}{t_1}\wedge\cdots\wedge\frac{dt_n}{t_n}
  \end{equation}
  for some homology class
  \(\tilde{C}\in\mathrm{H}_n((\mathbb{C}^{\ast})^{n}\cap(X\setminus\cup_{i=1}^{r}Y_{x_{i}}),\mathbb{Z})\).
  One can check directly that \eqref{equation:complete:intersection:toric:cy2}
  is a solution to \(M_{A,\beta_0}\) where
  \begin{equation*}
    \beta_0=(-1,\ldots,-1,0,\ldots,0)\in\mathbb{C}^r\times\mathbb{C}^n.
  \end{equation*}
  This integral can be related to the homology of the affine part of the
  complete intersection \(Y_{x_1} \cap \cdots \cap Y_{x_r}\) via the well-known
  residue method.
\end{remark}

\begin{situation}%
  \label{situation:hypothesis}
  \textbf{Hypothesis.}
  Throughout this note, we assume that the vector
  \(-\beta_{0}\) lies in the \emph{interior} of the cone generated by
  column vectors of \(A\).
  This hypothesis is needed to ensure that we can apply a theorem of
  Reichelt~\cite{2014-Reichelt-laurent-polynomials-GKZ-hypergeometric-systems-and-mixed-hodge-modules}.
  See Lemma~\ref{lemma:semi-nonresonant}.
\end{situation}

This hypothesis is not as technical as it may look. For instance, when \(A\) is
from ``toric geometry'' as in
Remark~\ref{remark:toric-cy-complete-intersections}, this hypothesis is
always verified as long as \(\mathbf{0}\) is an interior point of \(\Delta\).
This is a consequence of the following lemma.

\begin{lemma}
  Given a lattice polytope \(\Delta\),
  not necessarily reflexive, with \(\mathbf{0}\in\mathrm{int}(\Delta)\) and
  a decomposition \(\Delta=\Delta_{1}+\cdots+\Delta_{r}\) with \(\Delta_{i}\)
  being integral lattice polytopes, let
  \(A\) be the integral matrix defined in
  Remark~\ref{remark:toric-cy-complete-intersections} from these data.
  Assume further that \(\mathbf{0}\in\Delta_{i}\) for all \(i\),
  then \(-\beta_{0}\) lies in the interior of the cone generated
  by column vectors of \(A\).
\end{lemma}

\begin{proof}
  We regard each column of \(A\) as a vector in \(\mathbb{R}^{r+n}\)
  and let \((x_{1},\ldots,x_{n+r})\) be the coordinates.
  Let \(\mathbb{R}_{\ge 0}A\) be the cone generated by the columns of \(A\).
  Looking at the intersection of \(\mathbb{R}_{\ge 0}A\) with the
  affine hyperplane \(H\) defined by \(x_{1}+\cdots+x_{r}=1\).
  Note that all the column vectors of \(A\)
  are sitting in this hyperplane.
  Then \(-\beta_{0}\) lies in the interior of \(\mathbb{R}_{\ge 0}A\)
  if and only if \(-\beta_{0}\slash r\) lies in
  the interior of the convex hull determined by the column vectors of \(A\).

  We want to prove that there exists an \(\epsilon>0\) such that
  \begin{equation*}
    B_{\epsilon}(-\beta_{0}/r)\cap H\subset \mathbb{R}_{\ge 0}A\cap H
  \end{equation*}
  (\(B_{\epsilon}(\rho)\) is the open ball of radius \(\epsilon\) centered at \(\rho\)).
  First note that any \(y\in \Delta\), there exist
  \(0\le t_{i}\le 1\) such that \(y=\sum_{i} t_{i}m_{i}\), where
  \(m_{i}\) are vertexes of \(\Delta\), since a polytope is
  the convex hull of its vertexes.
  Since \(\mathbf{0}\in\mathrm{int}(\Delta)\), by rescaling \(\Delta\), for
  any \(R>0\) we may find a
  \(\delta>0\) such that for each \(y\in B_{\delta}(\mathbf{0})\) we have
  \(y=\sum_{i} t_{i}m_{i}\) with \(0\le t_{i}<1/R\).

  Now we can choose \(\epsilon\) small enough such that
  \begin{equation*}
  z=(z_{1},\ldots,z_{r+n})\in B_{\epsilon}(-\beta_{0}/r)\cap H\Rightarrow
  (z_{r+1},\ldots,z_{r+n})\in B_{\delta}(\mathbf{0}).
  \end{equation*}
  We can find \(0\le t_{i}\le 1/R\) such that \((z_{r+1},\ldots,z_{r+n})=\sum t_{i}m_{i}\).
  Let \(a_{i}\) be the colume vector of \(A\) corresponding to \(m_{i}\).
  Then
  \begin{equation*}
  \sum t_{i} a_{i} = (q_{1},\ldots,q_{r},z_{r+1},\ldots,z_{r+n}),
  \end{equation*}
  where \(q_{i}=\sum_{j:m_{j}\in\Delta_{i}} t_{j}\). We can take \(R\) large
  such that \(q_{i}<1/r\) for all \(i=1,\ldots,r\).
  Let \(b_{i}\) be the column vector corresponding to \(\mathbf{0}\in\Delta_{i}\). We have
  \begin{equation*}
    z = \sum t_{i}a_{i} + \sum_{i=1}^{r} (1/r-q_{i})b_{i} \in\mathbb{R}_{\ge 0}A.
  \end{equation*}
  This completes the proof.
\end{proof}

Now we can state our main theorem.

\begin{theorem}
  \label{eq:main}
  Adapt the notation introduced in~\ref{situation:toric-notation}.
  Under Hypothesis~\ref{situation:hypothesis} and
  the assumptions on \(A\) made in~\ref{situation:toric-notation}, if
  \begin{equation*}
  \beta_0=(-1,\ldots,-1,0,\ldots,0)\in\mathbb{C}^r\times\mathbb{C}^n,
  \end{equation*}
  then for \emph{any} toric resolution \(X \to X'\) (in fact, we only need to
  assume that \(X\) to be a smooth algebraic variety containing \(T\) as an open
  dense subset, and the morphism \(X \to X'\) restricts to the identity on
  \(T\)) and any \(x\in V^{\vee}\), we have
	\begin{equation*}
  \mathrm{Sol}^0(M_{A,\beta_0},\widehat{\mathcal{O}}_{V^{\vee},x}) = \mathrm{H}_{n}(U_{x}, U_{x} \cap D)
  = \mathrm{H}_{n}(U'_{x},U'_{x} \cap D').
	\end{equation*}
\end{theorem}

Here
\(\mathrm{Sol}^{0}(-)=R^{0}\mathcal{H}om_{\mathscr{D}_{V^{\vee}}}(-,\widehat{\mathcal{O}}_{V^{\vee},x})\)
is the underived solution functor of \(\mathscr{D}\)-modules which outputs the
set of formal power series solutions of a \(\mathscr{D}\)-module around a point
\(x\).

When \(X'\) itself is smooth, \(r=1\) and
\(Y_{x}\) are Calabi--Yau, (\ref{eq:main}) was proved by
Huang~et.~al.~\cite{2016-Huang-Lian-Yau-Zhu-chain-integral-solutions-to-tautological-systems} using the general theory
of~\cite{2016-Huang-Lian-Zhu-period-integrals-and-the-riemann-hilbert-correspondence}. The
said theory, as we understood,
requires the smoothness hypothesis in a crucial way.

An~Huang has informed us, in a private communication, that he can
prove~\eqref{eq:main} for mirror quintics. Related to our result is Jie~Zhou's
work~\cite{2017-Zhou-gkz-hypergeometric-series-for-the-hesse-pencil-chain-integrals-and-orbifold-singularities}.
He gives an explicit description of the
solutions to the \(A\)-hypergeometric system associated with the matrix
\begin{equation*}
A =
\begin{bmatrix}
  1 & 1  & 1  & 1  \\
  0 & 2  & -1 & -1 \\
  0 & -1 &  2 & -1
\end{bmatrix},
\end{equation*}
using relative homology classes on the Hesse pencil of elliptic
curves.

\subsection*{An application}
Let \(\Delta_i\), \(1\le i\le r\), be polytopes such that
\(\mathbf{0}\in\Delta_i\) for all \(i\) and
\(\Delta=\Delta_1+\cdots+\Delta_r\). Assume that \(\mathbf{0}\) is an interior
point of \(\Delta\).
The integral points in \(\Delta_i\) define an integral matrix \(A\) as in
Remark~\ref{remark:toric-cy-complete-intersections}.
We retain the notation in Situation~\ref{situation:toric-notation}
and let \(\beta_0=(-1,\ldots,-1,0,\ldots,0)\in\mathbb{C}^r\times\mathbb{C}^n\).

\begin{corollary}[Existence of rank 1 points]
  \label{situation:corollaries}
  Let notation be as above. There exists a point \(x\in V^{\vee}\) such that
  \(\mathrm{Sol}^0(M_{A,\beta_0},\widehat{\mathcal{O}}_{V^{\vee},x})\) is of
  rank one.
\end{corollary}

\begin{proof}
  We choose \(x\) to be the section corresponding to the lattice points
  \((0,\ldots,0)\), which exists since \(\mathbf{0}\in\Delta_{i}\).
  The hypersurface \(Y_{x}\) is just the union of all toric divisors and
  \(U_{x}=T\). The assertion follows since \(\mathrm{H}_n(T)\) is of rank one.
\end{proof}

The \(r=1\) case was essentially proved by
Hosono--Lian--Yau~\cite{1997-Hosono-Lian-Yau-maximal-degeneracy-points-of-gkz-systems}
using another approach. The idea that one can use a certain cohomological
interpretation to prove the existence of rank 1 point is due to
Huang--Lian--Zhu~\cite{2016-Huang-Lian-Zhu-period-integrals-and-the-riemann-hilbert-correspondence}.

\subsection*{Remark on the injectivity of \texorpdfstring{\(\tau\)}{}}
In Situation~\ref{situation:toric-notation} we assumed the map \(\tau\) is
injective. When \(\beta=\beta_0\), we demonstrate
below how to achieve this by modifying \(A\) without changing the ambient
\(\mathscr{D}\)-module  \(M_{A,\beta_0}\)

Let \(A\) be an \((r+n)\times N\) integral matrix. Assume that \(A\) satisfies
all the hypotheses in Situation~\ref{situation:toric-notation}, except the item
(5), i.e., we do not assume that the toric mapping \(\tau\) is 
injective. We impose the following condition on \(A\):
\begin{equation}
  \label{eq:condition-star}
  \tag{\(\ast\)}
  \text{For each \(i\),
  there exists \(1\le j\le N_i\) such that \(w_{i,j}=(0,\ldots,0)^t\).}
\end{equation}
Let \(C\) be the \(n\times N\) integral matrix obtained
by removing the first \(r\) rows from \(A\).
Regarding \(C\) as a map \(\mathbb{C}^{N}\to\mathbb{C}^{n}\),
\(\mathrm{im}(C)\) is subgroup in \(\mathbb{Z}^{n}\) of finite index.
Let \(\{\mathrm{f}_{1},\ldots,\mathrm{f}_{n}\}\) be a \(\mathbb{Z}\)-basis of \(\mathrm{im}(C)\)
and \(\{\mathrm{e}_{1},\ldots,\mathrm{e}_{n}\}\) be the standard \(\mathbb{Z}\)-basis of \(\mathbb{Z}^{n}\).
They are related by
\begin{equation*}
\mathrm{f}_{i} = \sum_{j=1}^{n} b_{ji} \mathrm{e}_{j},~i=1,\ldots,n.
\end{equation*}
\(B=(b_{ji})\) is clearly an integral matrix and \(B\in\mathrm{GL}_{n}(\mathbb{Q})\).
Note that \(B^{-1}\) may not be integral. However, \(B^{-1}C\) is an integral matrix.
To see this, if \(c_{k}\) is the \(k\)\textsuperscript{th} column of \(C\), we write
\begin{equation*}
c_{k} = \sum_{l=1}^{n} d_{lk} \mathrm{f}_{l}.
\end{equation*}
\(d_{lk}\in\mathbb{Z}\) since \(\{\mathrm{f}_{1},\ldots,\mathrm{f}_{n}\}\) is a \(\mathbb{Z}\)-basis of
\(\mathrm{im}(C)\). Put \(D=(d_{lk})\). From the relation \(C=BD\), we deduce that \(B^{-1}C=D\) is integral and the
columns of \(D\) generate \(\mathbb{Z}^{n}\).
Now the condition \eqref{eq:condition-star} ensures that the columns of the matrix
\begin{equation*}
\begin{bmatrix}
      I_r & 0 \\
      0 & B^{-1}\\
    \end{bmatrix}\cdot A
\end{equation*}
generate \(\mathbb{Z}^{r+n}\). We thus deduce that

\begin{lemma}
  \label{lemma:sublattice}
  Let \(A\) be as above. Assume further that \(A\) 
  has property~\eqref{eq:condition-star}.
  Let \(\tau\) be the morphism defined by \(A\).
  We denote by \(\left.\tau\right|_{T}\) the pullback of \(\tau\)
  via \(T\to (\mathbb{C}^{\ast})^{r}\times T\), \(t\mapsto (\mathbf{1},t)\).
  Let \(B\in\mathrm{GL}_n(\mathbb{Q})\) such that
  the columns of \(B\) form
  an integral basis for the image torus \(T':=\mathrm{im}(\left.\tau\right|_T)\).
  Define the \((r+n)\times(r+n)\) matrix \(R\)
  \begin{equation*}
    R =
    \begin{bmatrix}
      I_r & 0 \\
      0 & B^{-1}\\
    \end{bmatrix},
  \end{equation*}
  Then the columns of \(RA\) generate \(\mathbb{Z}^{r+n}\) as a \(\mathbb{Z}\)-module.
\end{lemma}

\begin{lemma}%
  \label{lemma:injective-gkz}
  Let \(A\) and \(R\) be as in Lemma~\ref{lemma:sublattice}. We have
  \(M_{A,\beta_0}=M_{RA,\beta_0}\).
\end{lemma}

\begin{proof}
For a pair of non-negative integral
vectors \(\nu_{\pm}\), \(A\nu_{+}=A\nu_{-}\) if and only if
\(RA\nu_{+}=RA\nu_{-}\). Thus we get the same collection of box operators.
Since the first \(r\) rows of \(A\) and \(RA\) are the same, the Euler operators
induced by the \((\mathbb{C}^{\ast})^{r}\)-action remain unchanged. The Euler
operators associated with \(T\) action and \(T'\) action differed by \(B\).
Since the character set \(\beta_0\) has last \(n\) entries equal to zero,
The Euler equations associated with \(T\) and \(T'\) are also the same.
\end{proof}

In \cite{1993-Aspinwall-Greene-Morrison-the-monomial-divisor-mirror-map},
Aspinwall, Greene, and Morrison introduced the ``simplified moduli space'' to approximate
the (polynomial) moduli space of Calabi--Yau hypersurfaces in toric varieties.
Given a lattice reflexive polytope \(\Delta\subset M_{\mathbb{R}}\simeq\mathbb{R}^{n}\),
the normal fan of \(\Delta\) determines an \(n\)-dimensional toric variety
with maximal torus \(T=(\mathbb{C}^{\ast})^{n}\).
Let \((\Delta\cap M)_{0}\) be
the set of integral points which do not lie in the interior of any
codimension one face of \(\Delta\).
We denote by \(\mathbb{C}^{\#(\Delta\cap M)_{0}}\) the vector space of
Laurent polynomials of the form
\begin{equation*}
\sum_{m\in (\Delta\cap M)_{0}} a_{m} t^{m},~a_{m}\in\mathbb{C}.
\end{equation*}
The \(T\)-action on \((\Delta\cap M)_{0}\) induces a \(T\)-action on
\(\mathbb{C}^{\#(\Delta\cap M)_{0}}\). Let \(\mathbb{C}^{\ast}\) act
on \(\mathbb{C}^{\#(\Delta\cap M)_{0}}\) by an overall scaling.
The ``simplified moduli space'' is defined by the GIT quotient
\(\mathbb{C}^{\#(\Delta\cap M)_{0}}\sslash (\mathbb{C}^{\ast}\times T)\).

Let \(A_{0}\) be the integral matrix whose columns
consist of the vectors of the form \((1,m)^{t}\) with \(m\in (\Delta\cap M)_{0}\).
The ``period integrals'' on the space
\(\mathbb{C}^{\#(\Delta\cap M)_{0}}\) as well as on its GIT quotient
\(\mathbb{C}^{\#(\Delta\cap M)_{0}}\sslash (\mathbb{C}^{\ast}\times T)\)
can be studied by the GKZ system
\(M_{A_{0},\beta_{0}}\) with \(\beta_{0}=(-1,0,\ldots,0)\) as usual.
The matrix \(A_{0}\) satisfies the condition~\eqref{eq:condition-star}
(with \(r=1\)).
We can apply our method in Lemma~\ref{lemma:sublattice} to
solve the GKZ system \(M_{A_{0},\beta_{0}}\). Here is an example.

\begin{instance}
  \label{example:simplified-moduli}
  Let \(\Delta\) be the convex hull of \((-1,-1)\), \((2,-1)\) and \((-1,2)\)
  in \(\mathbb{R}^{2}\). \(\Delta\) can be identified with the section polytope
  of \(-K_{\mathbb{P}^{2}}\). \((\Delta\cap M)_{0}\) consists of four integral points:
  \((-1,-1)\), \((2,-1)\), \((-1,2)\) and \((0,0)\). The vector space of Laurent polynomials
  \(\mathbb{C}^{\#(\Delta\cap M)_{0}}\) is the Dwork family in \(\mathbb{P}^2\):
  \begin{equation*}
    a_1z_1^3+a_2z_2^3+a_3z_3^3+a_4z_1z_2z_3=0,~[z_1:z_2:z_3]\in\mathbb{P}^2.
  \end{equation*}
  Let \(\beta_{0}=(-1,0,0,0)\) as before.
  The corresponding \(A_{0}\) matrix in the GKZ system is
  	\begin{equation*}
	A_{0} =
	\begin{bmatrix}
  	1 & 1  & 1  & 1  \\
  	0 & 2  & -1 & -1 \\
  	0 & -1 &  2 & -1
	\end{bmatrix}.
	\end{equation*}
	The columns of \(A_{0}\) do not generate \(\mathbb{Z}^3\).
	However, we can perform row operations on the last two rows of \(A_{0}\) to get
	\begin{equation*}
	A_{0}' =
	\begin{bmatrix}
  	1 & 1  & 1  & 1  \\
  	0 & 0  & 1 & -1 \\
  	0 & -1 &  1 & 0
	\end{bmatrix}.
	\end{equation*}
	The matrices \(B\) and \(B^{-1}\) in this case are
	\begin{equation*}
  B =
  \begin{bmatrix}
  1 & -2\\
  1 & 1
  \end{bmatrix},~
	B^{-1} =
	\begin{bmatrix}
	1/3 & 2/3\\
	-1/3 & 1/3
	\end{bmatrix}.
	\end{equation*}
	and we have
	\begin{equation*}
	\begin{bmatrix}
	1 & 0\\
	0 & B^{-1}
	\end{bmatrix}
	\begin{bmatrix}
  	1 & 1  & 1  & 1  \\
  	0 & 2  & -1 & -1 \\
  	0 & -1 &  2 & -1
	\end{bmatrix}
	=
	\begin{bmatrix}
  	1 & 1  & 1  & 1  \\
  	0 & 0  & 1 & -1 \\
  	0 & -1 &  1 & 0
	\end{bmatrix}
	\end{equation*}
	From the discussion above,
	\begin{equation*}
		\mathrm{Sol}^0(M_{A_{0},\beta_{0}},\widehat{\mathcal{O}}_{V^{\vee},x})
		=\mathrm{Sol}^0(M_{A_{0}',\beta_{0}},\widehat{\mathcal{O}}_{V^{\vee},x})=
		\mathrm{H}_{n}(U_{x}, U_{x} \cap D)
	\end{equation*}
	with \(X\) being a resolution of ``mirror \(\mathbb{P}^2\)'' and \(U_{x}=X\setminus Y_{x}\).
\end{instance}

\subsection*{Independence of relative homology.}
  \label{situation:sanity}
As we have been asked several
times, we should explain the (trivial) fact that the relative homology groups
displayed above are independent of the choice of the resolution, i.e., the
second equality in~\eqref{eq:main}.
We choose to use the language of sheaves to show this.

Let \(W\) be a (locally quasi-compact, Hausdorff)
topological space (in our case \(W = U_{x} \cap T\)). Assume that there is a commutative diagram
\begin{equation*}
  \begin{tikzcd}
    W \ar[r,"j_1"] \ar[d,equal] & Z_1 \ar[d,"{\varphi}"] \\
    W \ar[r,swap,"{j_2}"] & Z_2
  \end{tikzcd}
\end{equation*}
of (locally quasi-compact, Hausdorff) topological spaces, in which \(j_1\) and
\(j_2\) are open embeddings and \(\varphi\) is proper. Then we know
\begin{equation*}
  \mathrm{H}_{m}(Z_i, Z_i - W) = R^{-m}\Gamma_c(Z_i,Rj_{i\ast}\omega_{W}),
  \quad i = 1, 2,
\end{equation*}
where \(\omega_W\) is the dualizing complex of \(W\) (this can be also served as
the \emph{definition} of the relative homology groups). But then
\begin{align*}
  R\Gamma_c(Z_2, Rj_{2\ast}\omega_W) &= R\Gamma_c(Z_2, R\varphi_{\ast}Rj_{1\ast}\omega_W) & \text{ by the commutativity} \\
  &= R\Gamma_c(Z_2, R\varphi_{!}Rj_{1\ast}\omega_W) & \text{by the properness of }\varphi \\
  &= R\Gamma_c(Z_1, Rj_{1\ast}\omega_W).
\end{align*}
This proves that the relative homology groups are independent of the choice of
the \(Z_i\)'s. Taking \(Z_i\) to be \(U_{x}\) and \(U'_{x}\) respectively proves the
independence of relative homology groups in~\eqref{eq:main}.

\section{\texorpdfstring{\(\beta_0\)}{beta-0} is semi-nonresonant}
\label{sec:semi-nonresonant}
The proof of~(\ref{eq:main}) contains an input: a comparison theorem between an
extraordinary Gauss--Manin system and the \(A\)-hypergeometric system, proved by
Reichelt~\cite{2014-Reichelt-laurent-polynomials-GKZ-hypergeometric-systems-and-mixed-hodge-modules}
(see also the results of Walther and
Schulze~\cite{2009-Schulze-Walther-hypergeometric-d-modules-and-twisted-gauss-manin-systems}).
In this section we verify that our parameter \(\beta_0\) is semi-nonresonant,
thus we can apply Reichelt's theorem.

\begin{definition}%
  \label{definition:semi-nonresonant}
  Let notation be as in Situation~\ref{situation:toric-notation}.
  We say \(\beta\) is \emph{semi-nonresonant} if
  \begin{equation*}
    \beta\notin \bigcup_{F} (\mathbb{Z}^{r+n}\cap \mathbb{Q}_{\ge 0}A) + \mathbb{C}F,
  \end{equation*}
  where the union is taken over all the faces \(F\) of \(A\). Recall that a face
  \(F\) of \(A\) is a subset of columns of \(A\) that minimizing some nonzero linear
  functional on the cone generated by \(A\). \(\mathbb{Q}_{\ge 0} A\) denotes
  the \(\mathbb{Q}_{\ge0}\)-span of the columns of \(A\) and \(\mathbb{C}F\) the
  \(\mathbb{C}\)-span of \(F\).
\end{definition}

\begin{theorem}[Reichelt]
  \label{eq:reichelt-theorem}
  In Situation~\ref{situation:toric-notation}, if \(\beta=(\beta_{i})\in\mathbb{C}^{r+n}\)
  is semi-nonresonant,
  then we have
  \begin{equation*}
    \mathrm{FT}(\tau_{!}\mathcal{O}_{(\mathbb{C}^{\ast})^{r}\times T}^{\beta})
    = M_{A,\beta}
  \end{equation*}
\end{theorem}
Here \(\mathrm{FT}\) stands for the Fourier--Laplace transform of
\(\mathscr{D}\)-modules and \(\mathcal{O}_{(\mathbb{C}^{\ast})^{r}\times T}^{\beta}\)
is the cyclic \(\mathscr{D}\)-module
\begin{equation*}
\mathscr{D}_{(\mathbb{C}^{\ast})^{r}\times T}\slash
\mathscr{D}_{(\mathbb{C}^{\ast})^{r}\times T}
\langle s_{i}\partial_{s_{i}}-\beta_{i},t_{j}\partial_{t_{j}}-\beta_{j+r}\colon
1\le i\le r,~1\le j\le n \rangle.
\end{equation*}
For the proof,
see~\cite{2014-Reichelt-laurent-polynomials-GKZ-hypergeometric-systems-and-mixed-hodge-modules}*{Proposition~1.14}.

The following lemma enables us to apply Reichelt's theorem in our setting.

\begin{lemma}%
  \label{lemma:semi-nonresonant}
  In situation~\ref{situation:toric-notation}, under
  Hypothesis~\ref{situation:hypothesis},
  the parameter \(\beta_0\) is semi-nonresonant.
\end{lemma}

\begin{proof}
  Suppose on the contrary that \(\beta_{0}\in \bigcup_{F} (\mathbb{Z}^{r+n}\cap \mathbb{Q}_{\ge 0}A) + \mathbb{C}F\).
  We can find \(m\in (\mathbb{Z}^{r+n}\cap \mathbb{Q}_{\ge 0}A)\), \(c\in\mathbb{C}\) and \(f\in F\)
  such that \(\beta_{0}=m+c\cdot f\). Then we have
  \begin{equation*}
    m + (-\beta_{0}) = m-\beta_{0} = - c\cdot f\in \mathbb{C}F.
  \end{equation*}
  By our Hypothesis~\ref{situation:hypothesis}, \(m+(-\beta_{0})\) is contained
  in the \emph{interior} of the cone generated by column vectors of \(A\), but
  \(-c \cdot f\) is in the \emph{boundary}. This is a contradiction: we pick
  a nonzero linear functional \(h:\mathbb{R}^{r+n}\to \mathbb{R}\) defining
  \(F\), we then deduce from the minimizing property that
  \(h(-c\cdot f) < h(m+(-\beta_{0}))\).
\end{proof}

\section{Functors on \texorpdfstring{\(\mathscr{D}\)}{D}-modules}
\label{sec:d-mod}
To fix notation used throughout this note, we recall some
  notions in algebraic \(\mathscr{D}\)-modules.

  Let \(X\) be a \emph{smooth} variety and \(\mathscr{D}_X\) be the sheaf of
  differential operators on \(X\). By a \(\mathscr{D}_X\)-module on \(X\) we
  always mean a \emph{left} \(\mathscr{D}_X\)-module. Let
  \(D^{b}_{h}(\mathscr{D}_X)\) be the bounded derived categories of
  \(\mathscr{D}\)-modules over \(X\) with \emph{holonomic} cohomology sheaves.
  Let \(D^{b}_{rh}(\mathscr{D}_X)\) be the derived category of complex of
  \(\mathscr{D}_X\)-modules with regular holonomic cohomology sheaves. One can
  define the \emph{duality functor}, denoted by \(\mathcal{M}\mapsto\mathbb{D}\mathcal{M}\), on
  \(D_{h}^b(\mathscr{D}_X)\). Let \(f:X\to Y\) be a morphism between smooth
  varieties. One can define the following functors
\begin{itemize}
\item For a complex \(\mathcal{M}\in D_{h}^b(\mathscr{D}_X)\), let
  \(f_+(\mathcal{M}):=Rf_\ast(\mathrm{dR}_{X/Y}(\mathcal{M}))\), where
  \(\mathrm{dR}_{X/Y}\) is the relative de Rham functor.
\item For a complex \(\mathcal{N}\in D_{h}^b(\mathscr{D}_Y)\), let
  \(f^!\mathcal{N}:=f^\ast\mathcal{N}[\dim X-\dim Y]\), where \(f^\ast\) is the
  derived pullback on the category of quasi-coherent \(\mathcal{O}_Y\)-modules.
\end{itemize}
Note that these functors can be defined on the category of
\(\mathscr{D}\)-modules without the holonomic condition. Nonetheless, all the
functors \(\mathbb{D}\), \(f_+\) and \(f^!\) preserve the holonomicity. We put
\begin{itemize}
	\item \(f^+:=\mathbb{D}_X f^!\mathbb{D}_Y\), and
	\item \(f_!:=\mathbb{D}_Y f_+\mathbb{D}_X\).
\end{itemize}
\(f^+\) is the left adjoint of \(f_+\) and \(f_!\) is the left adjoint of \(f^!\).

When \(f\) is a smooth morphism, or more generally non-characteristic with respect
to a holonomic \(\mathscr{D}\)-module \(\mathcal{M}\), we have
\(f^{\ast}\mathcal{M}=f^{!}\mathcal{M}[\dim Y-\dim X]=f^{+}\mathcal{M}[\dim
X-\dim Y]\). Finally, given a cartisian diagram
\begin{equation*}
  \begin{tikzcd}
    X' \ar[r,"g'"] \ar[d,"f'"] & X \ar[d,"f"] \\
    Y' \ar[r,"g"] & Y,
  \end{tikzcd}
\end{equation*}
with all varieties are smooth, then we have
\begin{equation*}
  g^{!} f_+ = f'_+{ g' }^!.
\end{equation*}

Let \(S\subset X\) be a (possibly singular) subscheme of \(X\) and \(\mathscr{I}_S\)
be the corresponding ideal sheaf. For a \(\mathcal{O}_X\)-module \(\mathcal{F}\) on
\(X\), we define
\begin{equation*}
  \Gamma_{[S]}(\mathcal{F}):= \varinjlim_k \mathcal{H}om_{\mathcal{O}_X}(\mathcal{O}_X/\mathscr{I}^k_S,\mathcal{F}).
\end{equation*}
The quasi-coherent \(\mathcal{O}_{X}\)-module \(\Gamma_{[S]}(\mathcal{F})\)
inherits a \(\mathscr{D}_X\)-module structure and we can consider its right
derived functor \(R\Gamma_{[S]}\). When \(\mathcal{M}\) is a complex with holonomic
cohomology sheaves, so is \(R\Gamma_{[S]}(\mathcal{M})\). Let \(j:X\setminus S\to
X\) be the open embedding. For \(\mathcal{M}\in D^b_{h}(\mathscr{D}_X)\) we have
the distinguished triangle
\begin{equation}\label{distinguished:triangle:singular:support}
  R\Gamma_{[S]}(\mathcal{M})\to \mathcal{M} \to j_+j^!\mathcal{M}\to R\Gamma_{[S]}(\mathcal{M})[1].
\end{equation}
Let \(i:S \to X\) be the closed embedding. In case \(S\) is \emph{smooth},
we have \(R\Gamma_{[S]}(\mathcal{M})\simeq i_+i^!\mathcal{M}\) and the
distinguished triangle \eqref{distinguished:triangle:singular:support} becomes
\begin{equation}
  i_+i^!\mathcal{M}\to \mathcal{M} \to j_+j^!\mathcal{M}\to.
\end{equation}
Therefore we shall sometimes abuse notation and use \(i_{+}i^!\) instead of
\(R\Gamma_{[S]}\) even when \(S\) is singular.

The proofs of the above said results can be found in
\cite{2003-Baldassarri-DAgnolo-on-dwork-cohomology-and-algebraic-d-modules}.

\section{Cohomology computation}
\label{sec:cohomology}
In this section, we explain how to use exponential twists to
compute cohomology groups. There are three topics:
\begin{itemize}
\item how to use an exponential twisted de~Rham cohomology (Dwork cohomology) to compute the
  cohomology of a complete intersection,
\item how to use an exponential twisted de~Rham cohomology to compute cohomology of the complement of
  anormal-crossing divisor, and finally
\item what is the sheaf-theoretic mechanism for computing relative cohomology.
\end{itemize}
The main result is Lemma~\ref{lemma:compute-union-complement-cohomology}.

\begin{definition}
Let \(\gamma:Z\to \mathbb{A}^1\) be a morphism between smooth algebraic
varieties. We define the \emph{exponential \(\mathscr{D}\)-module} on \(Z\) to be
\begin{equation}
\exp(\gamma):=\gamma^\ast(\mathscr{D}_{\mathbb{A}^1}/(\partial_t-1))=\gamma^!(\mathscr{D}_{\mathbb{A}^1}/(\partial_t-1))[1-\dim Z].
\end{equation}
This is a holonomic, but generally irregular, \(\mathscr{D}\)-module on \(Z\).
\end{definition}

\begin{situation}\label{Fourier:Laplace:transform:Dwork:complex:diagram}
  Let \(X\) be a smooth algebraic variety and \(\pi:E\to X\) be a rank \(r\) vector bundle.
  Let \(\sigma:X\to E^\vee\) be a section of the dual bundle and \(\Sigma\) be the
  reduced zero scheme of \(\sigma\). We consider the following commutative
  diagram, which will be used frequently throughout this note.
  \begin{equation*}
    \begin{tikzcd}
      X \ar{r}{\iota} & E^{\vee} & E \times_X E^{\vee} \ar[swap]{l}{\mathrm{pr}_2} \ar{r}{\gamma} & \mathbb{A}^{1}. \\
      \Sigma \ar{u}{i} \ar[swap]{r}{i} & X \ar{u}{\sigma} & E \ar{l}{\pi} \ar{u}{\varepsilon} \ar[ur, "F"'] &
    \end{tikzcd}
  \end{equation*}
  In this diagram, \(\iota\) the embedding of the zero section, \(\gamma\) is the
  natural pairing, \(i\) is the closed embedding, \(\varepsilon\) is the
  pullback of \(\sigma\), and \(F=\gamma\circ\varepsilon\).
\end{situation}

\begin{definition}
  Let notation be as in
  Situation~\ref{Fourier:Laplace:transform:Dwork:complex:diagram}.
  For a holomonic complex \(\mathcal{M}\in D^b_h(\mathscr{D}_X)\), we define
  \emph{the Dwork complex} of \(\mathcal{M}\), denoted by
  \(\mathrm{Dw}_{E/X}(\mathcal{M})\) or simply \(\mathrm{Dw}(\mathcal{M})\), to
  be
  \begin{equation}
    \mathrm{Dw}(\mathcal{M}):=\pi_+(\pi^{\ast}\mathcal{M}\otimes \exp(F)).
  \end{equation}
  A priori, \(\mathrm{Dw}_{E/X}(\mathcal{M})\) seems to be irregular, but the
  following theorem implies it is in fact regular.
\end{definition}

\begin{theorem}\label{Dwork:theorem}
  In Situation~\ref{Fourier:Laplace:transform:Dwork:complex:diagram}, we have an isomorphism
  \begin{equation*}
    a(\mathcal{M}) : \mathrm{Dw}(\mathcal{M})\simeq R\Gamma_{[\Sigma]}(\mathcal{M})[r]\simeq
    R\Gamma_{[\Sigma]}(\mathcal{O}_X)\otimes_{\mathcal{O}_{X}} \mathcal{M}[r]
  \end{equation*}
  for any object \(\mathcal{M} \in D^b_h(\mathscr{D}_X)\).
\end{theorem}

The theorem was obtained by many people in various situations:
by Katz~\cite{1968-Katz-on-the-differential-equations-satisfied-by-period-matrices} when \(X\) is affine and \(E\)
is the trivial line bundle,
by Adolphson--Sperber~\cite{2000-Adolphson-Sperber-dwork-cohomology-de-rham-cohomology-and-hypergeometric-functions}
when \(X\) is the affine space but \(E\) can have higher rank, and by
Dimca--Maaref--Sabbah--Saito~\cite{2000-Dimca-Maaref-Sabbah-Saito-dwork-cohomology-and-algebraic-d-modules},
independently Baldassarri--D'Agnolo~\cite{2003-Baldassarri-DAgnolo-on-dwork-cohomology-and-algebraic-d-modules},
when \(X\) and \(E\) are both general.
The reader is referred to~\cite{2003-Baldassarri-DAgnolo-on-dwork-cohomology-and-algebraic-d-modules}
for a proof of the theorem as its proof is more related to our note.

For us, what is important is that the isomorphy for a general \(\mathcal{M}\) is obtained
from the isomorphy for \(\mathcal{O}_{X}\) by tensoring with
\(\mathrm{id}_{\mathcal{M}}\). This point is clear from Baldassarri--D'Agnolo's proof
we just cited, i.e., the following diagram commutes
\begin{equation*}
  \begin{tikzcd}
    \mathrm{Dw}(\mathcal{M}) \ar[d,"{a(\mathcal{M})}"] \ar[r,"{\sim}"] & \mathrm{Dw}(\mathcal{O}_X) \otimes \mathcal{M}
    \ar[d,"{a(\mathcal{O}_X)\otimes\mathrm{id}_{\mathcal{M}}}"]\\
    R\Gamma_{[\Sigma]}(\mathcal{M})[r] \ar[r,"{\sim}"] &    R\Gamma_{[\Sigma]}(\mathcal{O}_X)[r] \otimes \mathcal{M}
  \end{tikzcd}
\end{equation*}
In the diagram the horizontal arrows are the canonical isomorphisms.

\begin{situation}%
  \label{situation:union-hypersurface-complement-cohomology}
  Next let us explain how to compute the cohomology of the complement of a
  normal crossing divisor using an exponential twist. Consider the following
  situation.
  \begin{itemize}
  \item Let \(X\) be a smooth algebraic variety.
  \item Let \(\mathcal{L}_{1},\ldots,\mathcal{L}_{r}\) be \(r\) invertible
    sheaves on \(X\).
  \item Let
    \(\mathbb{L}_{m}=\mathit{Spec}(\mathrm{Sym}^{\bullet}\mathcal{L}_{m}^{\vee})\), for
    \(m=1,\ldots,r\). Let \(\mathbb{L}_{m}^{\vee}\) be the dual of
    \(\mathbb{L}_{m}\).
  \item Let
    \(\mathbb{L}=\mathbb{L}_{1}\times_{X}\cdots\times_{X}\mathbb{L}_{r}\). This
    is a vector bundle of rank \(r\). Let
    \(\pi:\mathbb{L}\to X\) be the bundle projection.
  \item Let \(\sigma_m \in \mathrm{H}^{0}(X,\mathcal{L}_{m}^{\vee})\).
    Let \(H_m = V(\sigma_m) \subset X\), and \(H = \bigcup_{m=1}^{r}H_{m}\).
    Let \(\rho: U = X \setminus H \to X\) be the open immersion.
    For a subset \(I \subset \{1,2,\ldots,r\}\),
    let \(H_I = \bigcap_{m\in I} H_m\). For an integer \(k\), define
    \(H_{(k)}=\coprod_{\#{I}=k}H_{I}\). Let \(i_{k}:H_{(k)} \to X\) be the
    natural morphism induced by the inclusions.
  \item The sections \(\sigma_m\) define a function \(F: \mathbb{L} \to \mathbb{A}^{1}\).
  \item Let \(Z_m \subset \mathbb{L}\) be the fiber product
    \begin{equation*}
      \mathbb{L}_{1} \times_X \cdots \times_{X} \mathbb{L}_{m-1} \times_{X} [0_{m}] \times_X \mathbb{L}_{m+1}\cdots \times_{X} \mathbb{L}_{r}
    \end{equation*}
    where \([0_{m}]\) is the zero section of \(\mathbb{L}_{m}\). Then \(Z_{m}\)
    are Cartier divisors on \(\mathbb{L}\).
  \item For \(I \subset \{1,2,\ldots,r\}\), let \(Z_{I}=\bigcap_{m\in I}Z_{m}\).
    Then \(Z_{I}\) are vector bundles over \(X\) of rank \(r - \#I\).
    Hence \(Z=Z_{1}\cup \ldots\cup{}Z_{m}\) is a simple normal divisor of \(\mathbb{L}\). Let
    \(Z_{(k)}=\coprod_{\#I=k} Z_{I}\). Let \(\iota_{k}: Z_{(k)} \to \mathbb{L}\)
    be the natural morphism.
  \item Let \(\theta: \mathbb{L}^{\circ} \to \mathbb{L}\) be the open immersion
    whose complement is \(Z\).
  \end{itemize}
\end{situation}

\begin{lemma}%
  \label{lemma:compute-union-complement-cohomology}
  Let notation be as in
  Situation~\ref{situation:union-hypersurface-complement-cohomology}. Let \(\mathcal{M}\)
  be a bounded complex of \(\mathscr{D}_{X}\)-modules with regular singular,
  holonomic cohomology sheaves.
  Assume that either \(r=1\), or \(r>1\) but
  assume further that the \(H_{i}\)'s meet transversely.
  Then there is an isomorphism
  \begin{equation*}
    \rho_{+}\rho^{!}\mathcal{M}
    \simeq \pi_{+}((\theta_{!}\theta^{!}(\pi^{\ast}\mathcal{M})) \otimes \exp(F)).
  \end{equation*}
\end{lemma}

\begin{proof}
Let \(\mathcal{K}\) be a regular holonomic complex of
\(\mathscr{D}_{\mathbb{L}}\)-modules. Using the covariant Riemann--Hilbert
correspondence, that \(Z\) is a normal crossing divisor, and results in
topology, we infer that the complex \(\theta_{!}\theta^{!}\mathcal{K}\) can be computed by
the following double complex
\begin{equation*}
\mathcal{K} \to \iota_{1+}\iota_{1}^{+}\mathcal{K} \to \iota_{2+}\iota_{2}^{+}\mathcal{K}
\to \cdots \to \iota_{r+}\iota_{r}^{+}\mathcal{K}.
\end{equation*}
Since \(\pi\) is affine, and since the complexes above underlie complexes of
quasi-coherent \(\mathcal{O}_{\mathbb{L}}\)-modules, \(\pi_{+}\) has no higher
direct images. Thus, if we take \(\mathcal{K} = \pi^{\ast}\mathcal{M}\), we see
\begin{equation*}
  \pi_{+}((\theta_{!}\theta^{!}(\pi^{\ast}\mathcal{M})) \otimes \exp(F))
\end{equation*}
is computed by the direct image
\begin{equation}
  \label{eq:compute-direct-image}
  \pi_{+}( \pi^{\ast}\mathcal{M}\otimes \exp(F)) \to
  \pi_{+}(\iota_{1+}\iota_{1}^{+}\pi^{\ast}\mathcal{M}  \otimes \exp(F)) \to  \cdots \to
  \pi_{+}(\iota_{r+}\iota_{r}^{+}\pi^{\ast}\mathcal{M} \otimes \exp(F)).
\end{equation}
In what follows, we shall compute the complexes
\(\pi_{+}(\iota_{k+}\iota_{k}^{+} \pi^{\ast}\mathcal{M} \otimes \exp(F))\).
We claim that there is an isomorphism
\begin{equation}
  \label{eq:compare-dw-with-cohomology-with-support}
  \pi_{+}(\iota_{k+}\iota_{k}^{+}\pi^{\ast}\mathcal{M}
  \otimes\exp(F))\simeq{i}_{k!}i_{k}^{!}\mathcal{M}[r].
\end{equation}

Using projection formula, we have
\begin{equation*}
\pi_{+}(\iota_{k+}\iota_{k}^{+}\pi^{\ast}\mathcal{M}\otimes\exp(F))
\simeq (\pi\circ\iota_{k})_{+}(\iota_{k}^{+}\pi^{\ast}\mathcal{M}\otimes\iota_{k}^{\ast}\exp(F)).
\end{equation*}
Since \(Z_{(k)}\) is a disjoint union, we can treat each \(Z_I\) separately.
We shall only consider the case \(I=\{r-k+1,\ldots,r\}\).
The proofs for the rest cases are the same,
but one needs to use different index sets.
The variety \(Z_{I}\) is the same as the bundle
\(\mathbb{L}_{1} \times_{X} \cdots \times_{X} \mathbb{L}_{r-k}\). Let \(\pi_{k}\)
be the bundle projection
\(\mathbb{L}_{1} \times_{X} \cdots \times_{X}\mathbb{L}_{r-k} \to X\).

To simplify the notation, set \(\nu_k = \iota_{k}|_{Z_{I}}\).
Note that
\(\pi \circ \nu_k = \pi_{k}\). Moreover, the
function \(F|_{Z_{I}}\) is identified with the function
\(F_{k}:\mathbb{L}_{1}\times_{X}\cdots\times_{X}\mathbb{L}_{r-k}\to\mathbb{A}^{1}\)
defined by the sections \(\sigma_{1},\ldots,\sigma_{r-k}\). Then on \(Z_{I}\) we have
\begin{align*}
  &(\pi\circ\iota_{k}|_{Z_{I}})_{+}(\iota_{k}^{+}\pi^{\ast}\mathcal{M}\otimes\iota_{k}^{\ast}\exp(F))|_{Z_{I}} \\
  \simeq&~(\pi\circ\nu_{k})_{+}(\nu_{k}^{+}\pi^{\ast}\mathcal{M} \otimes \nu_{k}^{\ast}\exp(F)) \\
  \simeq&~\pi_{k+}(\nu_{k}^{+}\pi^{+}\mathcal{M}[r] \otimes \exp(F_k)) \\
  \simeq&~\pi_{k+}(\pi_{k}^{\ast}\mathcal{M}[k] \otimes \exp(F_k)) \\
  \simeq&~\mathrm{Dw}_{\mathbb{L}_{1} \times_{X} \cdots \times_{X}\mathbb{L}_{r-k}/X}(\mathcal{M})[k]
  \xrightarrow{\sim} R\Gamma_{[H_1 \cap \cdots \cap H_{r-k}]}(\mathcal{M})[r].
\end{align*}
In the last step we used the isomorphism provided by Theorem~\ref{Dwork:theorem}.
This completes the proof of~\eqref{eq:compare-dw-with-cohomology-with-support}.
Thus, the direct image
\(\pi_{+}((\theta_{!}\theta^{!}(\pi^{\ast}\mathcal{M}))\otimes\exp(F))\) is computed by the
a double complex of the form
\begin{equation}\label{eq:compute-pushed-cohomology}
i_{r!}i_{r}^{!}\mathcal{M} \xrightarrow{a_r(\mathcal{M})}
i_{r-1,!}i_{r-1}^{!} \mathcal{M} \xrightarrow{a_{r-1}(\mathcal{M})} \cdots
\to i_{1!}i_{1}^{!}\mathcal{M} \xrightarrow{a_1(\mathcal{M})} i_{0!}i_{0}^{!}\mathcal{M} = \mathcal{M}.
\end{equation}
The arrows \(a_{m}(\mathcal{M})\) are obtained by pushing forward the arrows
in~\eqref{eq:compute-direct-image}. On the other hand, there are canonical arrows
\begin{equation*}
b_m(\mathcal{M}) : i_{m!}i_{m}^{!}\mathcal{M} \to i_{m-1,!}i_{m-1}^{!}\mathcal{M}
\end{equation*}
[These canonical arrows are from the distinguished triangle
\begin{equation*}
u_{+}u^{!}\mathcal{K} \to \mathcal{K} \to v_{+}v^{!}\mathcal{K} \to
\end{equation*}
associated with a closed embedding \(u\) of spaces (whose complement open
immersion denoted by \(v\)).]

We do not know whether \(a_m(\mathcal{M})=b_{m}(\mathcal{M})\). They probably
do equal.
But we only need to know that they differ by a constant on each direct summand.
To see this, we first assume \(\mathcal{M}=\mathcal{O}_{X}\). In this case, on
each \(H_{I}\), we know that
\(\mathrm{Hom}_{\mathscr{D}_{H_I}}(\mathcal{O}_{H_I},\mathcal{O}_{H_{I}})=\mathbb{C}^{c}\)
where \(c\) is the number of connected components of \(H_{I}\). Let \(u: H_{I}
\to H_{J}\) be the inclusion for some \(J\) such that \(\#J-\#I=1\). Then
\(a_{m}(\mathcal{O}_X)\) has a direct summand which is induced by the shift of
an arrow
\begin{equation*}
u_{!}u^{!}\mathcal{O}_{H_{J}} \to \mathcal{O}_{H_{J}}.
\end{equation*}
But by the previous equality, such an arrow, if nonzero, is unique up to a
nonzero scalar on each component. Since \(a_{m}(\mathcal{M})\) is obtained by
\(a_{m}(\mathcal{O}_X)\otimes\mathrm{id}_{\mathcal{M}}\) by the proof of
Theorem~\ref{Dwork:theorem}, we conclude that \(a_{m}(\mathcal{M})\) agrees with the
canonical one up to a scalar on each of its direct summand.

To conclude, we can deduce from the above that the double
complex~\eqref{eq:compute-pushed-cohomology} is isomorphic to the double complex
\begin{equation}\label{eq:final-result}
  i_{r!}i_{r}^{!}\mathcal{M} \xrightarrow{b_r(\mathcal{M})} i_{r-1,!}i_{r-1}^{!} \mathcal{M}
  \xrightarrow{b_{r-1}(\mathcal{M})} \cdots \to i_{1!}i_{1}^{!}\mathcal{M}
  \xrightarrow{b_1(\mathcal{M})} i_{0!}i_{0}^{!}\mathcal{M} = \mathcal{M}.
\end{equation}
When \(r=1\), we have won. Assume below that \(r>1\).
Applying the duality, we see \eqref{eq:final-result} is the dual of the complex
\begin{equation*}
  i_{r+}i_{r}^{+}\mathbb{D}\mathcal{M}  \leftarrow \cdots
  \leftarrow i_{1+}i_{1}^{+}\mathbb{D}\mathcal{M} \leftarrow \mathbb{D}\mathcal{M}
\end{equation*}
which is the complex that computes \(\rho_{!}\rho^{!}\mathbb{D}\mathcal{M}\). This can be
seen, again, by applying the covariant Riemann--Hilbert correspondence. Under
Riemann--Hilbert correspondence the complex becomes the standard
Mayer--Vietories type complex computing the cohomology of a normal crossing
variety --- this is the reason that we assumed, when \(r>1\), that \(\cup H_i\)
is a normal crossing divisor. Taking Verdier dual again, we conclude
that~\eqref{eq:final-result} computes \(\rho_{+}\rho^{+}\mathcal{M}\). This completes the
proof.
\end{proof}

\begin{situation}%
  \label{situation:relative-cohomology}
  \textbf{Relative cohomology.}~We review the mechanism of computing relative
  cohomology in terms of the language of \(\mathscr{D}\)-modules (via the
  Riemann--Hilbert correspondence).

  Let \(X\) be a smooth, proper, algebraic
  variety. Let \(b: T \to X\) be an affine open immersion with a complement
  divisor \(D\), possibly singular. Let \(Y\) be a Cartier divisor on \(X\) with
  complement \(U\). Assume that \(Y\) is smooth. Let \(Y_T = Y \cap T\),
  \(Y_{D}=Y \cap D\), etc. Let \(\mathcal{M}\in D^{b}(\mathscr{D}_{X})\).
  We form the following diagram
  \begin{equation}
    \label{eq:relative-situation}
    \begin{tikzcd}
      Y_T \ar[r,"i_T"]  \ar[d,"a"] & T \ar[d,"b"] & U_{T} \ar[l,swap,"\rho_T"] \ar[d,"c"] \\
      Y \ar[r,"i"] & X & U \ar[l,swap,"\rho"] \\
      Y_D \ar[u,"u"] \ar[r,"i_D"] & D\ar[u,"v"] & U_D\ar[l,swap,"\rho_D"]
      \ar[u,"w"]
    \end{tikzcd}.
  \end{equation}
  There is a complex of \(\mathscr{D}\)-modules that ``computes'' the relative cohomology of
  \(\rho^!\mathcal{M}\) on \(U\) with respect to \(U_D\) (in the sense that it
  corresponds to the complex whose cohomology is relative cohomology, via the
  Riemann--Hilbert correspondence): this is \(c_!c^!(\rho^!\mathcal{M})\).

  Let us explain this if the reader is not familiar to this formalism. In fact
  if \(\rho^!\mathcal{M}\) corresponds to a constructible complex \(F\) under the covariant
  Riemann--Hilbert functor, then there is a distinguished triangle
  \begin{equation*}
    c_{!}c^{!}F \to F \to w_{\ast}w^{\ast}F \to
  \end{equation*}
  (taking \(R\Gamma(U,-)\) recovers the usual long exact sequence of sheaf
  cohomology groups). The corresponding version of the above distinguished triangle in
  the language of \(\mathscr{D}\)-modules is
  \begin{equation*}
    c_! c^! (\rho^{!}\mathcal{M}) \to \rho^{!}\mathcal{M} \to w_+ w^+ (\rho^!\mathcal{M}) \to.
  \end{equation*}

Note that \(\rho\) is an open embedding, which implies
\(\rho^{!}=\rho^{\ast}\) and \(\rho_{T}^{!}=\rho_{T}^{\ast}\).
Manipulating with base change functors we have
 \begin{equation*}
   c_!c^!\rho^!\mathcal{M} =  c_!\rho_T^!b^!\mathcal{M} =\rho^!b_!b^!\mathcal{M}.
 \end{equation*}

 We can put the discussion above in the relative setting. Assume that all schemes
 in the diagram~\eqref{eq:relative-situation} are schemes over a smooth
 \(\mathbb{C}\)-variety \(B\) and let \(p: X \to B\) be the structure morphism.
 Then the ``variation'' of the relative cohomology groups (up to a shift) of the
 pair \((U, U\cap D)\) is computed by the complex
 \begin{equation}
   \label{eq:rh-relative-cohomology}
   Rp_{\ast}(X, R\rho_\ast \rho^{\ast} b_{!}\mathbb{C}[\dim T]) = \mathrm{RH}(p_+\rho_+\rho^!b_!\mathcal{O}_T),
 \end{equation}
 where \(\mathrm{RH}\) stands for the \emph{covariant} Riemann--Hilbert
 correspondence. Recall that its Verdier dual is the derived solution complex:
 \((\text{Verdier duality}) \circ \mathrm{RH} = \mathrm{Sol}\).
\end{situation}

\section{A lemma}
\label{sec:transition}
In this section, we prove a lemma which will be used to explain how the Dwork
complexes (and its siblings used in the previous section) can be related to the
Fourier--Laplace transform of a certain complexes.

\begin{situation}
\label{transition:lemma:situation}
\textbf{Setting.}~Retain the notation set up in
Situation~\ref{situation:union-hypersurface-complement-cohomology}. In addition,
we make the following assumptions.
\begin{itemize}
\item Assume that \(X\) is proper of pure dimensional \(n\).
\item Assume that each \(\mathcal{L}_i^\vee\) is generated by its global sections.
  Put \(V_i=\mathrm{H}^0(X,\mathcal{L}_i^\vee)^\vee\). We have morphisms
  \(X\to\mathbb{P}(V_i)\) and their product \(X\to \prod_{i=1}^r \mathbb{P}(V_i)\).
  Let \(V=V_1\times\cdots\times V_r\).
\item We will regard each \(V_i\) as an algebraic variety, which is
  \(\mathrm{Spec}(\mathrm{Sym}^\bullet(V_i^\vee))\).
\item For each \(i\), we denote by \(\pi_i\) the
  projection morphism \(\pi_i:\mathbb{L}_i\to X\),
\item For each \(i\), we have a morphism \(b_i:\mathbb{L}_i\to V_i\), which
  contracts the zero section to the origin. This can be obtained as follows.
  \(V_{i}^{\vee}\otimes\mathcal{O}_{X}\to \mathcal{L}_{i}^{\vee}\) gives a
  morphism of algebraic varieties \(\mathbb{L}_{i} \to V_{i}\times X\). Then
  \(b_{i}\) is the composition \(\mathbb{L}_{i}\to V_{i}\times X\to V_{i}\). Let
  \(b:\mathbb{L}\to V\) be their product.
\item Let \(\mathbb{L}_i^\circ\) be the subset of \(\mathbb{L}_i\) with its zero
  section removed. Thus
  \(\mathbb{L}_1^\circ\times_X\cdots\times_X\mathbb{L}_r^\circ\) is just
  \(\mathbb{L}^\circ\). Let \(\theta_i:\mathbb{L}_i^\circ\to\mathbb{L}_i\) be the
  open inclusion. Thus \(\theta:\mathbb{L}^\circ\to\mathbb{L}\) is the relative
  product of \(\theta_{i}\).
\item Let \(\widetilde{\mathbb{L}}\) be the pullback of \(\mathbb{L}\) via the
  natural projection \(X\times V^\vee\to X\). \(\widetilde{\mathbb{L}}^\vee\) is
  defined in the same manner. Let \(\tilde{\sigma}:X\times V^\vee\to
  \widetilde{\mathbb{L}}^\vee\) be the universal section. We will write
  \(\tilde{\sigma}=(\tilde{\sigma}_1,\ldots,\tilde{\sigma}_r)\).
  \(Y_i=\{\tilde{\sigma}_i=0\}\), \(1\le i\le r\),
  are the \emph{universal hypersurface} determined by \(\tilde{\sigma}\).
  Note that \(\cup_{i=1}^r Y_i\)
  is a simple normal crossing divisor.
\item All the morphisms obtained via the pullback along \(X\times V^\vee\to X\)
  are denoted by the same symbol with a `tilde'. For instance, the map
  \(\widetilde{\mathbb{L}}\to X\times V^\vee\) will be denoted by \(\tilde{\pi}\)
  according to our convention.
\item Put \(\iota=b\circ\theta:\mathbb{L}^\circ\to V\).
\end{itemize}
These data form the following commutative diagram (cf.~\S\ref{Fourier:Laplace:transform:Dwork:complex:diagram}).
  \begin{equation*}
    \begin{tikzcd}
      \widetilde{\mathbb{L}}^{\vee} & \widetilde{\mathbb{L}}^{\vee} \times_{X\times V^\vee} \widetilde{\mathbb{L}} \ar[swap]{l}{\mathrm{pr}_1} \ar{r}{\gamma} & \mathbb{A}^{1} \\
      X\times V^\vee \ar{u}{\widetilde{\sigma}} & \widetilde{\mathbb{L}} \ar{l}{\widetilde{\pi}} \ar{u}{\widetilde{\varepsilon}} \ar[ur, "F"']
    \end{tikzcd}
  \end{equation*}
In this diagram,
\begin{itemize}
	\item \(\widetilde{\varepsilon}\) is the pullback of \(\widetilde{\sigma}\),
	\item \(\gamma\) is the canonical dual pairing, and
	\item \(F=\gamma\circ\widetilde{\varepsilon}\).
\end{itemize}
\end{situation}

\begin{lemma}\label{lemma:transition}
  Let notation be as in \S\ref{transition:lemma:situation}. Let \(\mathcal{M}\) be
  a holonomic complex of \(\mathscr{D}\)-modules on \(\mathbb{L}^\circ\). Then
  \begin{equation}
    \mathrm{FT}(\iota_!\mathcal{M})=\mathrm{pr}_{V^\vee+}(\mathrm{pr}^!_{\mathbb{L}}\theta_!\mathcal{M}\otimes \exp(F))[-\dim V].
  \end{equation}
  Here \(\mathrm{pr}_{V^\vee}:\mathbb{L}\times V^\vee\to V^\vee\) and
  \(\mathrm{pr}_{\mathbb{L}}:\mathbb{L}\times V^\vee\to \mathbb{L}\) are
  projections.
\end{lemma}

\begin{proof}
  Look at the following commutative diagram
  \begin{equation*}
    \begin{tikzcd}
      \widetilde{\mathbb{L}}^{\circ} \ar[equal]{d} \ar{r}{\widetilde{\theta}} & \widetilde{\mathbb{L}} \ar[equal]{d}  &  &\\
      V^{\vee} \times \mathbb{L}^{\circ} \ar{r}{\widetilde{\theta}} \ar{d}{\mathrm{pr}_{\mathbb{L}^{\circ}}} & V^{\vee} \times \mathbb{L}  \ar{r}{\mathrm{Id}\times b} \ar{d}{\mathrm{pr}_{\mathbb{L}}} & V^{\vee}\times V \ar{d}\ar{d}{\mathrm{pr}_{V}} \ar{r}{\text{can}} & \mathbb{A}^1 \\
      \mathbb{L}^{\circ} \ar[bend right]{rr}{\iota}\ar{r}{\theta} & \mathbb{L}
      \ar{r}{b} & V &
    \end{tikzcd}.
  \end{equation*}
  Since \(b\) is proper, \(b_+ = b_!\). Thus
  by the base change theorem, and the
  commutativity of the entire lower rectangle, we have
  \begin{equation*}
    \mathrm{pr}_{V}^{!}\iota_{!}\mathcal{M}
    = \mathrm{pr}_{V}^{!} b_+ \theta_{!} \mathcal{M}
    = (\mathrm{Id}\times b)_+ \mathrm{pr}_{\mathbb{L}}^{!}\theta_{!}\mathcal{M}.
  \end{equation*}
  Hence
  \begin{align*}
    \mathrm{FT}(\iota_{!}\mathcal{M})
    &= \mathrm{pr}_{V^{\vee}+} (\mathrm{pr}_{V}^{!}\iota_{!}\mathcal{M} \otimes \exp\mathrm{can}) [-\dim V]\\
    &= (\mathrm{pr}_{V^{\vee}+} \circ (\mathrm{Id}\times b))_+ (\mathrm{pr}_{\mathbb{L}}^{!}\theta_{!}\mathcal{M} \otimes \exp\mathrm{can})[-\dim V] \\
    &= \mathrm{pr}_{V^{\vee}+}(\mathrm{pr}_{\mathbb{L}}^{!}\theta_{!}\mathcal{M} \otimes \exp(F))[-\dim V],
  \end{align*}
  as claimed.
\end{proof}

\section{Proof of the main theorem}

We prove our main theorem in this section.

Given an \((r+n)\times N\) integral matrix \(A\),
the columns of \(A\) give rise to a morphism
\(\tau: (\mathbb{C}^\ast)^r\times T\to \prod_{i=1}^r \mathbb{P}V_i\) as
in Situation~\ref{situation:toric-notation}.
Let \(X'\subset\mathbb{P}V_1\times\cdots\times\mathbb{P}V_r\) be the closure of the image.
It is a possibly singular toric variety.
Choose a toric desingularization \(X\to X'\).
Each \(X\to \mathbb{P}V_i\) determines a base point free line bundle \(\mathcal{L}_i^\vee\).
We are now in the situation of~\ref{transition:lemma:situation}.
Let us retain the notation there.

We have a sequence of maps.
\begin{equation*}
  \begin{tikzcd}
    (\mathbb{C}^{\ast})^r \times T \ar[rrr,bend left,swap,"\tau"] \ar[r,"j"] & \mathbb{L}^{\circ} \ar[rr,bend right,"\iota"] \ar[r,"\theta"]& \mathbb{L} \ar[r,"b"] & V
  \end{tikzcd}
\end{equation*}
On the one hand, owing to our hypothesis on \(A\),
Reichelt's result in~\cite{2014-Reichelt-laurent-polynomials-GKZ-hypergeometric-systems-and-mixed-hodge-modules}*{Proposition~1.14} implies
\begin{equation}
\mathrm{FT}(\tau_!\mathcal{O}_{(\mathbb{C}^{\ast})^r \times T})=M_{A,\beta}.
\end{equation}
On the other hand, applying Lemma~\ref{lemma:transition} to
\(\mathcal{M}=j_!\mathcal{O}_{(\mathbb{C}^{\ast})^r \times T}\), we obtain
\begin{equation}\label{Fourier:transform:hypergeometric:equation:to:Dwork}
\mathrm{FT}(\iota_!j_!\mathcal{O}_{(\mathbb{C}^{\ast})^r \times T})=\mathrm{pr}_{V^\vee+}(\mathrm{pr}^!_{\mathbb{L}}\theta_!j_!\mathcal{O}_{(\mathbb{C}^{\ast})^r \times T}\otimes \exp(F))[-\dim V].
\end{equation}
To proceed, we consider the following commutative diagram.
\begin{equation*}
  \begin{tikzcd}
   ((\mathbb{C}^{\ast})^r \times T)\times V^\vee \ar[d,"\mathrm{pr}_2"'] \ar[r,"\tilde{j}"] & \widetilde{\mathbb{L}}^{\circ}  \ar[r,"\tilde{\theta}"] \ar[d,swap,"\mathrm{pr}_{\mathbb{L}^{\circ}}"] & \widetilde{\mathbb{L}} \ar[d, "\mathrm{pr}_{\mathbb{L}}"] \\
  (\mathbb{C}^{\ast})^r \times T \ar[r,"j"]  & \mathbb{L}^{\circ} \ar[r,"\theta"] & \mathbb{L}.
  \end{tikzcd}
\end{equation*}
Note the all vertical maps are smooth.
We have \(\mathrm{pr}^\ast_\bullet=\mathrm{pr}^!_\bullet[-\dim V^\vee]=\mathrm{pr}^+_\bullet[\dim V^\vee]\).
An iterated application of the projection formula to this diagram yields
\begin{align}
\label{Fourier:to:Dwork:projection:formula}
\begin{split}
\mathrm{pr}^+_{\mathbb{L}}&\theta_!j_!\mathcal{O}_{(\mathbb{C}^{\ast})^r \times T}[\dim V^\vee]\\
&=\tilde{\theta}_!\mathrm{pr}^+_{\mathbb{L}^\circ}j_!\mathcal{O}_{(\mathbb{C}^{\ast})^r \times T}[\dim V^\vee]\\
&=\tilde{\theta}_!\tilde{j}_!\mathrm{pr}^+_2\mathcal{O}_{(\mathbb{C}^{\ast})^r \times T}[\dim V^\vee]\\
&=\tilde{\theta}_!\tilde{j}_!\mathrm{pr}^\ast_2\mathcal{O}_{(\mathbb{C}^{\ast})^r \times T}=\tilde{\theta}_! \tilde{j}_!\mathcal{O}_{((\mathbb{C}^{\ast})^r \times T)\times V^\vee}.
\end{split}
\end{align}
To compare this with the objects on \(X\),
we identify \(((\mathbb{C}^{\ast})^r \times T)\times V^\vee\) with \(\widetilde{\mathbb{L}}^\circ|_T\)
and look at the following commutative diagram.
\begin{equation*}
  \begin{tikzcd}
  \widetilde{\mathbb{L}}^\circ|_T \ar[d,"\tilde{j}"'] \ar[r,"\tilde{\theta}_T"] & \widetilde{\mathbb{L}}|_T \ar[r,"\tilde{\pi}_T"] \ar[d,swap,"\alpha"] & T\times V^\vee \ar[d, "\beta"] \\
  \widetilde{\mathbb{L}}^\circ \ar[r,"\tilde{\theta}"]  & \widetilde{\mathbb{L}} \ar[r,"\tilde{\pi}"] & X\times V^\vee.
  \end{tikzcd}
\end{equation*}
Since
\(\mathcal{O}_{\widetilde{\mathbb{L}}^\circ|_T}=(\tilde{\theta}_T)^+\mathcal{O}_{\widetilde{\mathbb{L}}|_T}=(\tilde{\theta}_T)^+(\tilde{\pi}_T)^+\mathcal{O}_{T\times V^{\vee}}[r]\),
the last quantity in \eqref{Fourier:to:Dwork:projection:formula} can be transformed into
\begin{align}
\begin{split}
\tilde{\theta}_!&\tilde{j}_!\mathcal{O}_{((\mathbb{C}^{\ast})^r \times T)\times V^\vee}\\
&=\tilde{\theta}_! \tilde{j}_!\mathcal{O}_{\widetilde{\mathbb{L}}^\circ|_T}\\
&=\tilde{\theta}_! \tilde{j}_!(\tilde{\theta}_T)^+(\tilde{\pi}_T)^+\mathcal{O}_{T\times V^\vee}[r]\\
&=\tilde{\theta}_! \tilde{\theta}^+\alpha_!\tilde{\pi}_T^+\mathcal{O}_{T\times V^\vee}[r]\\
&=\tilde{\theta}_! \tilde{\theta}^+\tilde{\pi}^+\beta_!\mathcal{O}_{T\times V^\vee}[r]\\
&=\tilde{\theta}_! \tilde{\theta}^!\tilde{\pi}^!\beta_!\mathcal{O}_{T\times V^\vee}[-r].
\end{split}
\end{align}
Plugging the displayed equation above
and \eqref{Fourier:to:Dwork:projection:formula} into \eqref{Fourier:transform:hypergeometric:equation:to:Dwork} yields
\begin{align}\label{Fourier:Dwork:complex:result}
\begin{split}
\mathrm{FT}(\iota_!j_!\mathcal{O}_{(\mathbb{C}^{\ast})^r \times T})&=\mathrm{pr}_{V^\vee+}(\mathrm{pr}^!_{\mathbb{L}}\theta_!j_!\mathcal{O}_{(\mathbb{C}^{\ast})^r \times T}\otimes \exp(F))[-\dim V]\\
&=\mathrm{pr}_{V^\vee+}(\tilde{\theta}_! \tilde{\theta}^!\tilde{\pi}^!\beta_!\mathcal{O}_{T\times V^\vee}\otimes \exp(F))[-r].\\
&=\mathrm{pr}_{V^\vee+}(\tilde{\theta}_! \tilde{\theta}^!\tilde{\pi}^\ast\beta_!\mathcal{O}_{T\times V^\vee}\otimes \exp(F)).
\end{split}
\end{align}
We recapitulate that the \(\mathrm{pr}_{V^\vee+}\) appeared above is the
projection from \(\mathbb{L} \times V^{\vee}\) to \(V^{\vee}\).

Now we can apply Lemma~\ref{lemma:compute-union-complement-cohomology}, by
\begin{itemize}
\item taking \(X\) in the lemma to be \(X\times V^\vee\),
\item taking \(\mathbb{L}\) in the lemma to be \(\widetilde{\mathbb{L}}\), and
\item taking \(\mathcal{M}\) in the lemma to be \(\beta_!\mathcal{O}_{T\times V^\vee}\).
\end{itemize}
We also use \(\mathrm{pr}_{V^{\vee}}\) to denote the projection from
\(X\times V^{\vee}\) to \(V^{\vee}\). Thus (by abuse of notation)
\(\mathrm{pr}_{V^{\vee}} = \mathrm{pr}_{V^{\vee}}\widetilde{\pi}\).
Thus Lemma~\ref{lemma:compute-union-complement-cohomology} implies that
\begin{equation*}
  \mathrm{pr}_{V^\vee+}(\tilde{\theta}_! \tilde{\theta}^!\tilde{\pi}^{\ast}\beta_!\mathcal{O}_{T\times V^\vee}\otimes \exp(F))
  = \mathrm{pr}_{V^{\vee}+}(\rho_{+}\rho^{!}\beta_{!}\mathcal{O}_{T\times V^{\vee}})
\end{equation*}
where \(\rho: X \times V^{\vee} - (\bigcup_{i=1}^{r} Y_{i}) \to X \times V^{\vee}\) is the open
immersion. By~\eqref{Fourier:Dwork:complex:result}, we have
\begin{equation*}
  M_{A,\beta}\simeq \mathrm{pr}_{V^{\vee}+}\rho_+\rho^! \beta_!\mathcal{O}_{T\times V^\vee}.
\end{equation*}

By~\eqref{eq:rh-relative-cohomology}, the
Riemann--Hilbert partner of the \(\mathscr{D}\)-module
\(\mathrm{pr}_{V^\vee+}\rho_+\rho^! \beta_!\mathcal{O}_{T\times V^\vee}\) is
\(R \mathrm{pr}_{V^\vee*}(\left.\beta_!\mathbb{C}[n]\right|_U)\).
For each \(x\in V^{\vee}\), let \(i_{x}\) be the closed immersion. Then we have
\begin{equation*}
  i_{x}^{!}R \mathrm{pr}_{V^\vee\ast}(\left.\beta_!\mathbb{C}[n]\right|_U)=R\Gamma(U_x,U_x\cap D)[n].
\end{equation*}
Taking the zeroth cohomology of the displayed complex and applying the duality between homology and cohomology,
we have
\begin{equation*}
\mathrm{Sol}^0(M_{A,\beta},\widehat{\mathcal{O}}_{V^{\vee},x}) = \mathrm{H}_{n}(U_x, U_x \cap D).
\end{equation*}

\end{document}